\RequirePackage{fix-cm}
\documentclass[smallextended,nospthms]{svjour3}       % onecolumn (second format)
\smartqed  % flush right qed marks, e.g. at end of proof
\usepackage{graphicx}
\usepackage{mathrsfs}
\usepackage{amsmath,amsfonts,amssymb,amsthm}
\usepackage{bbm}
\usepackage{color,xcolor}
\usepackage{verbatim}
\usepackage{enumerate}

\newcommand{\N}{\mathbb N}

\newcommand{\Z}{\mathbb Z}

\newcommand{\abs}[1]{\left\vert#1\right\vert}
\newcommand{\set}[1]{\left\{#1\right\}}
\newcommand{\eps}{\varepsilon}

\newcommand{\h}{\mathbf{h}}

\newcommand{\CC}{\mathcal{C}}

\newcommand{\CT}{\mathcal{T}}

\newcommand{\CS}{\mathcal{S}}

\renewcommand{\emptyset}{\varnothing}
\renewcommand{\tilde}{\widetilde}

\newcommand{\sd}{\bigtriangleup}

\renewcommand{\epsilon}{\varepsilon}

\renewcommand{\geq}{\geqslant}

\newtheorem{thm}{Theorem}[section]
\newtheorem*{thm*}{Theorem}
\newtheorem{lem}[thm]{Lemma}
\newtheorem*{lem*}{Lemma}

\newtheorem*{cor*}{Corollary}

\newtheorem*{prop*}{Proposition}
\theoremstyle{definition}
\newtheorem{defn}[thm]{Definition}
\newtheorem*{defn*}{Definition}
\theoremstyle{remark}
\newtheorem{rem}[thm]{Remark}
\newtheorem*{rem*}{Remark}

\newtheorem*{example*}{Example}

\newtheorem*{que*}{Question}

\begin{document}

\title{Tilings of amenable groups%\thanks{Grants or other notes
}


\author{Tomasz Downarowicz         \and
        Dawid Huczek \and
				Guohua Zhang
}


\institute{Tomasz Downarowicz \at
              Institute of Mathematics of the Polish Academy of Science\\ \'Sniadeckich 8, 00-956 Warszawa, Poland \\ and \\Institute of Mathematics and Computer Science \\ Wroclaw University of Technology \\ Wybrze\.ze Wyspia\'nskiego 27, 50-370 Wroc\l aw, Poland
              \email{downar@pwr.edu.pl}           %  \\
           \and
           Dawid Huczek \at
              Institute of Mathematics and Computer Science \\ Wroclaw University of Technology \\ Wybrze\.ze Wyspia\'nskiego 27, 50-370 Wroc\l aw, Poland
							\email{dawid.huczek@pwr.edu.pl}
						\and
					 Guohua Zhang \at
					 School of Mathematical Sciences and LMNS \\ Fudan University and Shanghai Center for Mathematical Sciences \\ Shanghai 200433, China
					\email{chiaths.zhang@gmail.com}
}

\date{Received: date / Accepted: date}

\maketitle

\begin{abstract}
We prove that for any infinite countable amenable group $G$, any $\eps>0$ and any finite subset $K\subset G$, there exists a tiling (partition of $G$ into finite ``tiles'' using only finitely many ``shapes''), where all the tiles are $(K,\eps)$-invariant. Moreover, our tiling has topological entropy zero (i.e., subexponential complexity of patterns). As an application, we construct a free action of $G$ (in the sense that the mappings, associated to different from unity elements of $G$, have no fixpoints), on a zero-dimensional space, and which has topological entropy zero.
\end{abstract}

\section{Introduction}
In this paper we solve a question about ``tileability'' of countable amenable groups using
finitely many tiles with good invariance properties. The problem was open for a long time, but it
was shaded by another, more difficult, problem about ``tileability'' using only one tile.

In 1980, D. Ornstein and B. Weiss \cite{OW1} announced that every such group can be \emph{$\epsilon$-tiled} using finitely many ``F\o lner tiles'' (sets belonging to the F\o lner sequence), i.e., covered up to $\epsilon$ (in terms of the invariant mean) by $\epsilon$-disjoint shifted copies of these finitely many sets. The authors admit that removing the inaccuracy in covering the group remains a problem.
Nonetheless, their construction (which appears in \cite{OW2}) became fundamental in the analysis of dynamical systems (both topological and measure-theoretic) with the action of countable amenable groups. It serves as a substitute of the Rohlin lemma, allowing to generalize to this case vast majority of entropic and ergodic theorems known for the actions of $\Z$ (see e.g. \cite{OW2}, \cite{WZ}, \cite{RW}, \cite{L}, \cite{Da} and the references therein). Later, in \cite{We}, B. Weiss showed that countable amenable groups from a large class admit a precise tiling by only one shape (monotile) belonging to a selected F\o lner sequence. This class includes all countable amenable linear groups and all countable residually finite amenable groups. But, to the best of our knowledge, the $\eps$-quasitilings have never been improved in full generality to precise tilings.
\smallskip

In an attempt to carry over the \emph{theory of symbolic extensions} (see \cite{BD}) to actions of amenable groups, we have encountered serious difficulties in applying quasitilings. The imprecision  in how they cover the group, although inessential in most other cases, destroys the construction of symbolic extensions. It turns out that not only we need the tilings to be precise, but also it is desirable that they have topological entropy zero (not just arbitrarily small). This has inspired us to abandon, at least for some time, symbolic extensions and focus on improving Ornstein-Weiss' machinery in the first place.

\smallskip
Below the reader will find a complete and self-contained construction of a precise tiling (i.e., partition) of any infinite countable amenable group by shifted copies of finitely many finite sets (which we call ``shapes'' as opposed to ``tiles'', as we call the elements of the partition), where the shapes (and hence the tiles) reveal arbitrarily good invariance properties under small shifts. In fact, we produce not one but countably many such partitions, which are \emph{congruent} (each is inscribed in the following one), each is  determined by the following one, and all of them have topological entropy zero (i.e., have subexponential complexity of patterns in large regions). Such sequence of tilings can be thought of as an analog of an odometer used frequently for $\Z$-actions (more generally, for actions of residually finite groups) to introduce a system of \emph{parsings}, i.e., partitions of orbits into finite \emph{blocks}, on which one can later perform various combinatorial manipulations.

Our construction starts with building an $\epsilon$-quasitiling almost identical to  that of Ornstein and Weiss, except that our meaning of $\epsilon$-covering is in terms of lower Banach density. Regardless of the similarity, we provide all details, mainly because of somewhat complicated lower Banach density estimates. This quasitiling is then modified four times: first it is made disjoint, next precisely covering (this step is the most novel), then, given a sequence of such tilings they are made congruent, finally they are brought to a form where each is a factor of the following. In each step we control their small topological entropy, so that after the last modification the entropy is killed completely to zero.

As an application we produce a free zero-entropy action of the group on a zero-dimensional space.

\medskip
We are convinced that exact tilings might simplify many proofs in the area of ergodic theory and topological dynamics for amenable group actions and perhaps allow for new developments. For instance, in addition to symbolic extensions, the creation of Brattelli diagrams might become possible for such actions, leading to better understanding of K-theory and orbit equivalence.

\section{Preliminaries}
\subsection{Basic notions}
Throughout this paper $G$ denotes an infinite countable \emph{amenable group}, i.e., a group in which there exists a sequence of finite sets $F_n\subset G$ (called a \emph{F\o lner sequence}, or the \emph{sequence of F\o lner sets}), such that for any $g\in G$ we have
\[
\lim_{n\to\infty}\frac{\abs{gF_n\sd F_n}}{\abs{F_n}}= 0,
\]
where $gF=\set{gf:f\in F}$, $\abs{\cdot}$ denotes the cardinality of a set, and $\sd$ is the symmetric difference. Since $G$ is infinite, the sequence $|F_n|$ tends to infinity. Without loss of generality (see \cite[Corollary 5.3]{N}) we can assume that the sets in the F\o lner sequence are symmetric and contain the unity.

\begin{defn}
If $T$ and $K$ are finite subsets of $G$ and $\eps<1$, we say that $T$ is \emph{$(K,\eps)$-invariant} if
\[\frac{\abs{KT\sd T}}{\abs{T}}< \eps,\]
where $KT=\set{gh : g \in K,h\in T}$.
\end{defn}

Observe that if $K$ contains the unity of $G$, then $(K,\eps)$-invariance is equivalent to the simpler condition
\[\abs{KT}< (1+\eps)\abs{T}.\]

The following fact is very easy to see, so we skip the proof.
\begin{lem}\label{folner}
A sequence of finite sets $(F_n)$ is a F\o lner sequence if and only if for every finite set $K$ and every $\eps>0$ the sets $F_n$ are eventually $(K,\eps)$-invariant.
\end{lem}

\begin{lem}\label{ben}
Let $K\subset G$ be a finite set and fix some $\eps>0$. There exists $\delta>0$ such that if
$T\subset G$ is $(K,\delta)$-invariant and $T'$ satisfies $\frac{|T'\triangle T|}{|T|}\le \delta$ then
$T'$ is $(K,\eps)$-invariant.
\end{lem}
\begin{proof}
We have $KT'\setminus KT\subset K(T'\setminus T)$ (and similarly for $T$ and $T'$ exchanged), so
$KT'\triangle KT= (KT'\setminus KT)\cup(KT\setminus KT')\subset K(T'\setminus T)\cup K(T\setminus T')=
K(T'\triangle T)$. Thus, by the triangle inequality for $|\cdot\triangle\cdot|$, we obtain
$$
\frac{|KT'\triangle T'|}{|T'|}\le
\frac{|KT'\triangle KT|+ |KT\triangle T|+|T'\triangle T|}{(1-\delta)|T|}\le  \frac{|K|\delta+\delta+\delta}{1-\delta}<\eps,
$$
if $\delta$ is sufficiently small.
\end{proof}

\begin{defn}
We say that $T'\subset T$ ($T$ finite) is a $(1-\eps)$-subset of $T$, if $\abs{T'}\geq(1-\eps)\abs{T}$.
\end{defn}

\begin{defn}
Let $K$ be a finite subset of $G$ and let $T\subset G$ be arbitrary. The \emph{$K$-core} of $T$, denoted as $T_K$, is the set $\set{g\in T: Kg\subset T}$ (this is the largest subset $T'\subset T$ satisfying $KT'\subset T$).
\end{defn}

The following fact is a fairly standard and easy exercise, nonetheless we give its proof for completeness.
\begin{lem}\label{estim}
For any $\eps>0$ and any finite $K\subset G$ there exists a $\delta$ (in fact $\delta = \frac\eps{|K|}$), such that if $T\subset G$ is finite and $(K,\delta)$-invariant then the $K$-core $T_K$ is a $(1-\eps)$-subset of $T$.
\end{lem}

\begin{proof} Note that $(K,\delta)$-invariance of $T$ implies that
$$(\forall g\in K)\ \ |gT\setminus T|<\delta|T|,$$
i.e.,
$$(\forall g\in K)\ \ |T\cap g^{-1}T| = |gT\cap T| >(1-\delta)|T|.
$$
This yields
$$|T_K|=\left|\bigcap_{g\in K}(T\cap g^{-1}T)\right|>(1-|K|\delta)|T|=(1-\eps)|T|.$$
\end{proof}

We end this subsection by recalling a standard combinatorial fact (see e.g. \cite[Lemma A.3.5]{Do}):
\begin{thm}[Marriage theorem (variant)]\label{marriage}
Let $Q$ be a countable set, let $A$ and $B$ be countable subsets of $Q$ and let $R$ be a relation between $B$ and $A$ such that for some positive integer $N$ the following two conditions hold:
\begin{itemize}
\item For any $b\in B$ the number of $a\in A$ such that $bRa$ is at least~$N$.
\item For any $a\in A$ the number of $b\in B$ such that $bRa$ is at most~$N$.
\end{itemize}
Then there exists an injective mapping $\phi$ from $B$ into $A$ such that $bR\phi(b)$, for all $b\in B$.
\end{thm}

\subsection{Lower and upper Banach density}

Below we present the notions of lower and upper Banach densities. In \cite{BBF} the reader will find a different, yet equivalent definition (the equivalence follows from Lemma \ref{bd} below and the fact that if $(F_n)$ is a F\o lner sequence, so is $(F_ng_n)$ for any choice of the $g_n$'s, see also \cite[formula (5)]{BBF}).

\begin{defn}
For $S\subset G$ and a finite $F\subset G$ denote
\[
\underline D_F(S)=\inf_{g\in G} \frac{|S\cap Fg|}{|F|}, \ \ \ \overline D_F(S)=\sup_{g\in G} \frac{|S\cap Fg|}{|F|}.
\]
If $(F_n)$ is a F\o lner sequence then we define two values
\[
\underline D(S)=\limsup_{n\to\infty} \underline D_{F_n}(S) \ \ \ \text{ and } \ \ \ \overline D(S)=\liminf_{n\to\infty} \overline D_{F_n}(S),
\]
which we call the \emph{lower} and \emph{upper} \emph{Banach densities} of $S$, respectively.
\end{defn}
Note that $\overline D(S) = 1- \underline D(G\setminus S)$. The following fact is fairly standard, we include its proof for completeness.

\begin{lem}\label{bd}
Regardless of the set $S$, the values of $\underline D(S)$ and $\overline D(S)$ do not depend on the
F\o lner sequence, the limits superior and inferior in the definition are in fact limits, moreover,
\begin{align*}
\underline D(S) &= \sup\{\underline D_F(S): F\subset G, F \text{ is finite}\}\ \ \ \text{ and } \\
\overline D(S) &= \,\inf\,\{\overline D_F(S): F\subset G, F \text{ is finite}\}\ge \underline D(S).
\end{align*}
\end{lem}

\begin{proof}
We will only show that
$$
\liminf_n\underline D_{F_n}(S)\ge\sup\{\underline D_F(S): F\subset G, F \text{ is finite}\}.
$$
This, and an analogous symmetric statement, clearly imply the assertion.
Fix some $\eps>0$ and let $F$ be a finite set such that
$$
\underline D_F(S)\ge \sup\{\underline D_{F'}(S): F'\subset G, F' \text{ is finite}\}-\eps.
$$
Let $n$ be so large that $F_n$ is $(F,\eps)$-invariant. Given $g\in G$, we have
$$
|S\cap Ffg|\ge \underline D_F(S)|F|,
$$
for every $f\in F_n$. This implies that there are at least $\underline D_F(S)|F||F_n|$ pairs $(f',f)$ with $f'\in F, f\in F_n$  such that
$f'fg\in S$. This in turn implies that there exists at least one $f'\in F$ for which
$$
|S\cap f'F_ng|\ge \underline D_F(S)|F_n|.
$$
Since $f'\in F$ and $F_n$ is $(F,\eps)$-invariant (and hence so is $F_ng$), we have
$$
|S\cap f'F_ng|\le|S\cap FF_ng|\le |S\cap F_ng|+\eps|F_n|,
$$
which yields
$$
|S\cap F_ng|\ge (\underline D_F(S)-\eps)|F_n|.
$$
We have proved that $\underline D_{F_n}(S)\ge \underline D_F(S)-\eps$, which ends the proof.
\end{proof}

\subsection{Some facts from symbolic dynamics}

Let $\Lambda$ be a finite set with discrete topology. There exists a standard action of $G$ on $\Lambda^G$ (called the \emph{shift} action), defined as follows: $(gx)(h)=x(hg)$. $\Lambda^G$ with the product topology and the shift action of $G$ becomes a zero-dimensional dynamical system, called the \emph{full shift over $\Lambda$}. A \emph{symbolic dynamical system over $\Lambda$} is any closed, $G$-invariant subset $X$ of the full shift. In the following paragraph we will need the notions of a
block and an associated cylinder set. Let $F\subset G$ be a finite set. By a \emph{block with domain $F$} we will understand any element $B\in \Lambda^F$. The \emph{cylinder corresponding to $B$} is the set
$$
[B] = \{x\in\Lambda^G:x|_F = B\}.
$$
If $X$ is a symbolic system, then the set $X_F = \{B\in\Lambda^F: X\cap [B]\neq\emptyset\}$ is interpreted as the family of blocks with domain $F$ which \emph{occur} in $X$.

Given a symbolic system $X$, we can construct its \emph{topological factor} in
form of a symbolic system $Y$ over another finite alphabet, say $\Delta$, applying so-called \emph{sliding block code with finite horizon}. Such a factor is determined by a mapping $\Pi:X_F\to\Delta$ called \emph{the code}, where the set $F$ is finite and called \emph{the horizon} (of the code).
The code $\Pi$ extends to a mapping $\pi:X\to \Delta^G$ by the formula $x\mapsto y$, where $y$ is given by
$$
y(g) = \Pi(gx|_F).
$$
Then the set $Y=\pi(X)$ is a closed, shift-invariant subset of $\Delta^G$. In practice, the fact that $Y$ is a topological factor of $X$ is verified by checking whether each element $y$ of $Y$ is \emph{determined} term by term by an $x\in X$ by means of some algorithm (procedure, reasoning) \emph{with finite horizon}, i.e., whether we can decide about the symbol $y(g)$ using only the information from
\begin{enumerate}
	\item $gx|_F$ (i.e., the symbolic contents of $x$ within the copy of $F$ shifted by $g$), and
	\item some \emph{finite bank of information} (i.e., the mapping $\Pi$, which in practice is the set of instructions how to deduce $\Pi(B)$ given $B\in\Lambda^F$).
\end{enumerate}
We will refer to this intuitive understanding of factor maps between symbolic systems several times.
In case $X$ is transitive with a transitive point $x$ (i.e., $X$ equals the orbit-closure of $x$), $Y$ is a topological factor of $X$ via a mapping $\pi$ and $y=\pi(x)$ (in which case $y$ is a transitive point for $Y$), then, abusing slightly the terminology, we will say that \emph{$y$ is a topological factor of $x$} (because in such case the ``finite bank of information'' about the factor map from $X$ to $Y$ can be reconstructed from the combinatorial relation between $x$ and $y$).
It is now not hard to see, that if $(x_n, y_n)$ is a sequence of pairs in the Cartesian square of $\Lambda^G$, converging to some $(x,y)$, $y_n$ is a topological factor of $x_n$ for every $n$, with a common (independent of $n$) coding horizon, then $y$ is a topological factor of $x$ (with the same horizon).

\smallskip
We should mention here another ingredient needed later in this work, connecting upper Banach density with
invariant measures. We will use it only for symbolic systems, although its generality is much wider. We
skip the proof, which is an immediate consequence of the Ergodic Theorem for amenable groups (\cite[Theorem 1.2]{L}).

\begin{lem}\label{measure}
Let $x\in\Lambda^G$ be a symbolic element, let $F\subset G$ be a finite set, and let $B\in\Lambda^F$ be a block. Then
$$
\mu([B])\le \overline D(\{g: gx|_F = B\}),
$$
for any invariant measure $\mu$ supported by the orbit closure of $x$.
\end{lem}

\smallskip

We will be using the notion of topological entropy for actions of amenable groups, but mainly in the case of symbolic systems. In this case it is completely analogous to that for symbolic systems with the action of $\mathbb{Z}$. In this note, by $\log$ we will mean $\log_2$.

\begin{defn}
The \emph{topological entropy} of a symbolic dynamical system $X$ is the limit
\[\h(X)=\lim_{n\to\infty}\frac{1}{\abs{F_n}}\log N(F_n),\]
where $N(F_n)$ is the \emph{$F_n$-complexity of $X$}, i.e., the cardinality $|X_{F_n}|$ of different blocks with domain $F_n$ occurring in $X$.
\end{defn}

The limit is known to exist and, by Orstein--Weiss Lemma, does not depend on the choice of the F\o lner sequence, see e.g. \cite[Theorem 6.1]{LW}. We will also need the following recent result (see \cite{DF}):

\begin{thm}\label{infimum}
The topological entropy of a symbolic dynamical system $X$ equals
\[\inf_F\frac{1}{\abs{F}}\log N(F),\]
where $F$ ranges over all finite subsets of $G$ and $N(F)$ is the $F$-complexity of $X$.
\end{thm}

It is well known, for actions of amenable groups, that if $Y$ is a topological factor of $X$ then
$\h(Y)\le\h(X)$.

\smallskip
If $X$ is transitive with a transitive point $x$ then $X_{F_n}$ coincides with the family of blocks $\{gx|_{F_n}:g\in G\}$, so that $\h(X)$ can be evaluated by examining $x$ only. In such case we will alternatively denote $\h(X)$ by $\h(x)$ and call it \emph{the entropy of $x$}. This convention will be used later to define \emph{entropy of a quasitiling}.

At some point we will also refer to measure-theoretic entropy and the variational principle. We choose to skip discussing these ingredients here; the necessary information is provided near where they are applied.

\smallskip
We need to say a few words about topological joinings of symbolic systems. We will restrict to transitive systems. Let $x$ and $y$ be two symbolic elements with possibly different alphabets, say $\Lambda$ and $\Delta$. The pair $(x,y)$ can be viewed as a symbolic element with the alphabet $\Lambda\times\Delta$. The shift-orbit closure of so understood $(x,y)$ is an example of a \emph{topological joining} $X\vee Y$ between the systems $X$ and $Y$ generated (as shift-orbit closures) by $x$ and $y$, respectively. There 
are many other joinings of $X$ and $Y$, (for instance, the direct product), but we will mostly use joinings of the above form. Both $X$ and $Y$ are topological factors of $X\vee Y$ (by projections), and we have the standard inequality
$$
\h(X\vee Y)\le \h(X)+\h(Y).
$$

At one occasion we will also use a countable product of subshifts. What we need to know about such products is that they are zero-dimensional and that the topological entropy equals the sum of the series of entropies of the component subshifts. These are standard facts and we skip their proofs.
\smallskip

\section{Quasitilings and tilings}

Our definitions given below are slightly different from the original ones in \cite{OW2}, however most of the
differences are inessential. Sometimes we will refer to quasitilings and tilings defined below as \emph{static} as opposed to \emph{dynamical}, which will be introduced later.

\begin{defn} A \emph{quasitiling} is determined by two objects:
\begin{enumerate}
	\item a finite collection $\CS(\CT)$ of finite subsets of $G$ containing the unity $e$,
	called \emph{the shapes}.
	\item a finite collection $\CC(\CT) = \{C(S):S\in\CS(\CT)\}$ of disjoint subsets of $G$, called \emph{center sets} (for the shapes).
\end{enumerate}
The quasitiling is then the family $\CT=\{(S,c):S\in\CS(\CT),c\in C(S)\}$. We require that the map $(S,c)\mapsto Sc$ is injective.\footnote{This requirement is stronger than asking that different tiles have different centers. Two tiles $Sc$ and $S'c'$ may be equal even though $c\neq c'$
(this is even possible when $S=S'$). However, when the tiles are disjoint, then the (stronger) requirement
follows automatically from the fact that the centers belong to the tiles.} Hence, by the \emph{tiles} of $\CT$ (denoted by the letter $T$) we will mean either the sets $Sc$ or the pairs $(S,c)$ (i.e., the tiles with defined centers), depending on the context.
\end{defn}

\begin{defn}\label{qt} Let $\eps\in[0,1)$ and $\alpha\in(0,1]$. A quasitiling $\CT$ is called
\begin{enumerate}
	\item \emph{$\eps$-disjoint} if there exists a mapping $T\mapsto T^\circ$ ($T\in\CT$) such that
	\begin{itemize}
	\item $T^\circ$ is a $(1-\eps)$-subset of $T$, and
	\item $T\neq T'\implies T^\circ\cap {T'}^\circ=\emptyset$;
	\end{itemize}
	\item \emph{$\alpha$-covering} if $\underline D(\bigcup\CT)\ge\alpha$;
	\item an (exact) \emph{tiling} if it is a partition of $G$.
\end{enumerate}
\end{defn}

\begin{rem}\label{centers}
Suppose $\CT$ is a quasitiling and $T\mapsto T^\heartsuit$ associates to the tiles of $\CT$ their
``modifications'', disjoint for different tiles. The ``modifications'' are assumed bounded in the
following sense: if $T=Sc$ is a tile (with shape $S$ and center $c$) then $T^\heartsuit c^{-1}$ is contained in some a priori given finite set $F$. Then the collection $\CT^\heartsuit = \{T^\heartsuit:T\in\CT\}$ gives rise to a new, disjoint quasitiling. We need, however, to redefine the centers so that they fall inside the new tiles (there are a priori no other restrictions). Once this is done, the collection of the new shapes is determined (and it is finite) and so are the corresponding center sets for each shape. By disjointness of the new tiles, all other requirements in the definition of quasitilings are fulfilled. However, a careless assignment of the centers may drastically enlarge the collection of shapes, and thus excessively increase the entropy of the quasitiling (see Definition \ref{dynam}). To avoid that, on every such occasion we will, by default, apply one deterministic procedure, as follows. Let us enumerate the set $F=\{g_1,g_2,\dots,g_l\}$. Now, if $T^\heartsuit$ is a tile obtained from $T=Sc$ then we set its center at the point $g\in T^\heartsuit$ such that $gc^{-1}$ has the smallest number among $T^\heartsuit c^{-1}$. This arrangement uses only the information from the quasitiling $\CT$, the mappings $T\mapsto T^\heartsuit$, and the ``finite bank'' of information about the ordering of the finite set $F$.
\end{rem}

\begin{lem}\label{cov1} Let $\CT$ be a $(1-\eps)$-covering quasitiling. Suppose $T\mapsto T^\heartsuit$
associates to the tiles of $\CT$ their $\alpha$-subsets, disjoint for different tiles.
Then the quasitiling $\CT^\heartsuit=\{T^\heartsuit:T\in\CT\}$ is $\alpha(1-\eps)$-covering.
In particular, for an $\eps$-disjoint, $(1-\eps)$-covering quasitiling $\CT$, the disjoint quasitiling $\CT^\circ$ (as in Definition \ref{qt}) is $(1-\eps)^2$-covering.
\end{lem}

\begin{proof} Denote $E=\bigcup\CS(\CT)$.
Let $\CT_F$ be the collection of tiles $T\in\CT$ entirely contained in some finite set $F$.
As easily verified, all other tiles $T\in\CT$ are disjoint with the $EE^{-1}$-core of $F$. Denote also $\CT_F^\heartsuit = \{T^\heartsuit:T\in\CT_F\}$ and
$$
\theta(F) = \left|F\cap\bigcup\CT_F^\heartsuit\right|.
$$
Clearly, $|F\cap\bigcup\CT_F|\le \theta(F)\frac1\alpha$. Thus and by the $(1-\eps)$-covering assumption, for any $\xi>0$ the following holds
$$
(1-\eps-\xi)|F|\le |F\cap\bigcup\CT|\le \theta(F)\tfrac1{\alpha}+|F\setminus F_{EE^{-1}}|\le \theta(F)\tfrac1\alpha+\xi|F|,
$$
if $F$ is ``sufficiently invariant" (see Lemma \ref{estim}), and then the same holds for any shifted set $Fg$. Thus
$$
\theta(Fg)\ge \alpha(1-\eps-2\xi)|F|.
$$
Since, by Lemma \ref{bd}, $\underline D(\bigcup\CT^\heartsuit)\ge \inf_g\frac{\theta(Fg)}{|F|}$, and $\xi$ is arbitrary, we obtain the assertion.
\end{proof}

\begin{defn}\label{dynam}
Every quasitiling $\CT$ can be represented in a symbolic form, as a point $x_\CT\in\Lambda^G$, with the alphabet $\Lambda = \CS(\CT)\cup\{0\}$, as follows: $x_\CT(g)=S \iff g\in C(S)$, $0$ otherwise. Let
$X_\CT$ be the orbit closure of $x_\CT$ under the shift action. This system is called the \emph{dynamical quasitiling} (generated by $\CT$). If $\CT$ is a tiling, we obtain a \emph{dynamical tiling}. 
According to our earlier convention, by the \emph{entropy} of $\CT$ (denoted as $\h(\CT)$) we will understand the topological entropy of $X_\CT$.
\end{defn}

\begin{rem}\label{dyntil}
Note that every element of $X_\CT$ represents a quasitiling with the same set of shapes as $\CT$. 
Moreover, if $\CT$ has any of the following properties: $\epsilon$-disjointness, disjointness, $(1-\epsilon)$-covering, being a tiling, then every quasitiling in $X_\CT$ has the same property.
\end{rem}

The following lemmas will be used in the entropy estimates of some quasitilings and tilings.

\begin{lem}\label{combinat}
There exists a function $\Theta:(0,1)\to(0,1)$, with $\lim_{\eps\to 0}\Theta(\eps) = 0$,
such that for any finite set $F\subset G$ and $\eps\in(0,1)$ the number of all subsets of $F$ with cardinality not
exceeding $\eps|F|$ is smaller than $2^{|F|\Theta(\eps)}$.
\end{lem}

The elementary proof is based on Stirling's formula.

\begin{lem}\label{entest}
For each positive integer $r$ and $\eps>0$ there exists a $\delta=\delta(r,\eps)>0$ for which the following statement holds:
Let $\CT$ be a $\frac14$-disjoint quasitiling of $G$. Suppose $\CS(\CT)$ can be divided into $r$ disjoint \emph{classes} $\CS_1,\CS_2,\dots,\CS_r$ such that, for each $i=1,2,\dots,r$, we have
\begin{align}
&s_i=\min\{|S|:S\in\CS_i\}\ge\frac1\delta, \\
&|\CS_i|\le 2^{\eps s_i}.
\end{align}
Then $\h(\CT)<3\eps$.
\end{lem}
\begin{proof}
Let $E=\bigcup\CS(\CT)$. Let $F=F_n$ (the element of the F\o lner sequence) for some large index $n$ (recall that then $|F|$ is also large).
We can assume that $F$ is $(E,\frac12)$-invariant. We will count the family $\{gx|_F:g\in G\}$, where $x$ is the symbolic representation of $\CT$.

Let $\CT_{Fg}$ denote the collection of tiles with centers in $Fg$. Notice that these tiles are contained in $EFg$. Further, let $\CT^\circ_{Fg}=\{T^\circ:T\in\CT_{Fg}\}$ (a disjoint selection of $\frac34$-subsets from each member of $\CT_{Fg}$). Denote by $\CT_{Fg,i}$ the subfamily of $\CT_{Fg}$ consisting of the tiles with shapes in $\CS_i$, and finally denote $\CT^\circ_{Fg,i}=\{T^\circ:T\in\CT_{Fg,i}\}$. The cardinalities of $\CT^\circ_{Fg,i}$ and $\CT_{Fg,i}$ are obviously the same, and the size of each $T^\circ\in\CT_{Fg,i}$
is at least $\frac34|T|$, hence at least $\frac34s_i$. By disjointness of the tiles $T^\circ$,
we have
$$
\sum_{i=1}^r s_i|\CT_{Fg,i}|\le \frac{4|EF|}{3}\le 2|F|.
$$
In particular, since each $s_i$ is not less than $\frac1\delta$, we get
$|\CT_{Fg}|\le 2\delta|F|$. Now replace each symbol $S$ in the symbolic representation of $x$ by the symbol $i$ such that $S\in\CS_i$.  The above procedure replaces $x$ by a new symbolic element $\hat x$, over the alphabet $\{0,1,2,\dots,r\}$. Since there are at most $2\delta|F|$ nonempty terms in the symbolic representation of $gx|_F$ (hence in $g\hat x|_F$), the number of possible blocks $g\hat x|_F$ does not exceed $2^{\Theta(2\delta)|F|}\cdot(r+1)^{2\delta|F|}$.

Next, we will bound the number of different blocks of the form $gx|_F$ which produce the same block of the form $g\hat x|_F$. So, fix some $g\in G$ and observe the block $g\hat x|_F$. Each symbol $i$ appears in $g\hat x|_F$ exactly $|\CT_{Fg,i}|$ times. This creates at most $|\CS_i|^{|\CT_{Fg,i}|}\le 2^{\eps s_i|\CT_{Fg,i}|}$ possibilities for the configurations of the symbols of $x$ replacing the symbols $i$. Jointly, there are at most
$$
\prod_{i=1}^r 2^{\eps s_i|\CT_{Fg,i}|} = 2^{\eps\sum_{i=1}^rs_i|\CT_{Fg,i}|}\le 2^{2\eps|F|}
$$
blocks of the form $gx|_F$ for each block of the form $g\hat x|_F$. Altogether, there are at most
$2^{\Theta(2\delta)|F|}\cdot(r+1)^{2\delta|F|}\cdot 2^{2\eps|F|}$ blocks of the form $gx|_F$, which, after taking logarithm, dividing by $|F|$ and letting $|F|\to\infty$, yields the entropy estimate
$$
\h(\CT)\le \Theta(2\delta)+ 2\delta\log(r+1)+2\eps.
$$
Since, by Lemma \ref{combinat}, $\Theta(2\delta)\to 0$ as $\delta\to 0$,
the assertion follows.
\end{proof}

We will be also using the following lemma. The easy proof resembles part of the proof of the preceding lemma (small entropy of $\hat x$).

\begin{lem}\label{symbolic_small_entropy}
For any finite set $\Lambda$ in which we select one element (and call it ``zero''), and $\eps>0$ there
exists a $\delta>0$ such that for any symbolic element $x$ with the alphabet $\Lambda$ and upper Banach density of non-zero symbols smaller than $\delta$, has entropy less than $\eps$.
\end{lem}

\section{The construction of an exact tiling}

In this section we construct an exact tiling of $G$ having ``well-invariant'' shapes and small entropy.
This is done in three steps: First we construct a quasitiling $\CT$ which is only $\eps$-disjoint and $(1-\eps)$-covering, then we modify it to a disjoint quasitiling $\CT^\circ$, which then is transformed
into an exact tiling $\CT^*$.
\smallskip

The following lemma is almost the same as \cite[I.\S 2. Theorem 6]{OW2}, differing from it in small details, in particular, our $(1-\eps)$-covering is defined in terms of lower Banach density.
Many ideas in our proof given below are the same as in \cite{OW2}, however we needed to add somewhat lengthy lower Banach density estimates.

\begin{lem}\label{quasitilings} Let $G$ be a countable amenable group with a F\o lner sequence $(F_n)$ of symmetric sets containing the unity. Given $\eps>0$, there exists a positive integer $r=r(\eps)$ such that for each positive integer $n_0$, there exists an $\eps$-disjoint, $(1-\eps)$-covering quasitiling $\CT$ of $G$ with $r$ shapes $\{F_{n_1},\dots,F_{n_r}\}$, where $n_0<n_1<\cdots<n_r$.
\end{lem}

\begin{proof}
Find $r$ such that $(1-\frac\eps2)^r<\eps$. This is going to be the cardinality of the family of shapes.
Choose integers $n_1=n_0+1, n_2,\dots, n_r$ so that they increase and for each pair of indices $j<i$, $j,i\in\{1,2,\dots, r\}$ the set $F_{n_i}$ is $(F_{n_j},\delta_j)$-invariant, where $\delta_j$ will be specified later. We let $\CS(\CT) = \{F_{n_j}: j=1,\dots,r\}$ be our family of shapes. With this choice, the assertions about the shapes and their number are fulfilled. It remains to construct the corresponding center sets $C (F_{n_j})$ so as to satisfy $\eps$-disjointness and $(1-\eps)$-covering of $\CT$.

We proceed by induction over $j$ decreasing from $r$ to 1.

Consider the collection of all such subsets $C\subset G$ that $\{F_{n_r}c:c\in C\}$ is an $\eps$-disjoint quasitiling (with one shape). As easily verified, when ordered by inclusion this collection satisfies the hypothesis of Zorn's Lemma (the disjoint family for the union of a chain can be found in a limit procedure).  Thus, there exists a maximal element in our collection (note that it is nonempty). We now pick one such maximal element and denote it $C_r$. At this point we define $C (F_{n_r})$ (the center set for the shape $F_{n_r}$) as $C_r$. We let $H_r = F_{n_r}C_r$ denote the part of the group covered by the so far constructed quasitiling. In order to estimate $\underline D(H_r)$ from below we will estimate $\underline D_{F_{n_r}}(H_r)$. If $g\in C_r$ then $F_{n_r}g$ is contained in $H_r$,
hence $\frac{|H_r\cap F_{n_r}g|}{|F_{n_r}|}=1$. For $g\notin C_r$, suppose that this ratio is strictly smaller than $\eps$. This implies that $F_{n_r}g$ can be added to the $\eps$-disjoint family $\{F_{n_r}c: c\in C_r\}$, contradicting the maximality of $C_r$.
That is, we have proved that for any $g\in G$, $\frac{|H_r\cap F_{n_r}g|}{|F_{n_r}|}\ge\eps$, i.e.,
that $\underline D_{F_{n_r}}(H_r)\ge\eps$.
Thus $\underline D(H_r)\ge\eps$ (which is strictly larger than $\frac\eps2=1-(1-\frac\eps2)^1)$.

Fix some $j\in\{1,2,\dots,r-1\}$ and suppose we have constructed an $\eps$-disjoint quasitiling $\{F_{n_i}c:j+1\le i\le r,\ c\in C_i\}$ (with $C_i$ abbreviating $C(F_{n_i})$, the center set for the shape $F_{n_i}$), whose union
$$
H_{j+1} = \bigcup_{i=j+1}^r\ F_{n_i}C_i
$$
has lower Banach density strictly larger than $1-(1-\frac\eps2)^{r-j}$ (this is our inductive hypothesis on $H_{j+1}$ and it is fulfilled for $H_r$). We need to go one step further in our ``decreasing induction'', i.e., add a center set $C_j$ for the shape $F_{n_j}$. Consider the collection of all subsets $C\subset G$ such that the family $\{F_{n_i}c:j+1\le i\le r,\ c\in C_i\}\cup\{F_{n_j}c:c\in C\}$ is an $\eps$-disjoint quasitiling (this includes that $C$ is disjoint from $\bigcup_{i=j+1}^rC_i$). As before, when ordered by inclusion, this collection satisfies the hypothesis of Zorn's Lemma. Thus, there exists a maximal element $C_j$ (this time possibly empty). We set $C (F_{n_j})= C_j$ and denote
$$
H_j = \bigcup_{i=j}^r\ F_{n_i}C_i.
$$

Our goal is to estimate from below the lower Banach density of $H_j$. By Lemma \ref{bd}, it suffices to estimate $\underline D_F(H_j)$ for just one finite set $F$ which we will define in a moment. Define $B=\Bigl(\bigcup_{i=j+1}^rF_{n_i}F^{-1}_{n_i}\Bigr)F_{n_j}$. Clearly, $B$ contains $F_{n_j}$
(hence the unity), and, as easily verified, it has the following property:
\begin{itemize}
\item
whenever $F^{-1}_{n_j}F_{n_i}c\cap A\neq\emptyset$, for some
$i\in\{j+1,\dots,r\}$, $c\in G$ and $A\subset G$, then $F_{n_i}c\subset BA$.
\end{itemize}
Let $n$ be so large that $F_n$ is $(B,\delta_j)$-invariant and that $\underline D_{F_n}(H_{j+1})>1-(1-\frac\eps2)^{r-j}$ (the latter is possible due to the assumption on $\underline D(H_{j+1})$). Now we define the aforementioned set $F$ as $F=F_{n_j}F_n$.
%Since $(F_{n_j}F_n)_{n\ge 1}$ is also a F\o lner sequence, we can additionally require $n$ to be so large that $\underline D_{F_{n_j}F_n}(H_{j+1})\ge \frac1{1+\delta_j}\underline D(H_{j+1})$.

Fix some $g\in G$ and define
$$
\alpha_g = \frac{|H_{j+1}\cap F_ng|}{|F_n|}  \text{ \ \ and \ \ } \beta_g = \frac{|H_{j+1}\cap BF_ng|}{|BF_n|}.
$$
Notice that
\begin{equation}\label{minus}
\alpha_g \ge \underline D_{F_n}(H_{j+1}) > 1-(1-\tfrac\eps2)^{r-j}.
\end{equation}
Also, we have
\begin{align}
\beta_g &\ge \frac{|H_{j+1}\cap F_ng|}{(1+\delta_j)|F_n|}=\frac{\alpha_g}{1+\delta_j}\label{zero} \ ,   \text{ \ \ and }\\
\beta_g &\le \frac{|H_{j+1}\cap F_ng|+|BF_ng\setminus F_ng|}{|F_n|}\le \alpha_g + \delta_j.
\end{align}
Note that since $F_{n_j}\subset B$ and $F_n$ is $(B,\delta_j)$-invariant, $F_n$ is automatically
$(F_{n_j},\delta_j)$-invariant. Thus
\begin{equation}\label{jeden}
\frac{|H_{j+1}\cap Fg|}{|F|}\ge \frac{|H_{j+1}\cap F_ng|}{(1+\delta_j)|F_n|}= \frac{\alpha_g}{1+\delta_j}\ge\frac{\beta_g-\delta_j}{1+\delta_j}.
\end{equation}

Consider only these finitely many component sets $F_{n_i}c$ of $H_{j+1}$ (i.e., with $i\in\{j+1,\dots,r\},\ c\in C_i$) for which $F^{-1}_{n_j}F_{n_i}c$ has a nonempty intersection with $F_ng$, and denote by $E_g$ the union of so selected components $F_{n_i}c$. By the property of $B$ (with $A=F_ng$), $E_g$ is a subset of $BF_ng$ (and also of $H_{j+1}$), so
\begin{equation}\label{niewiem}
|E_g|\le |H_{j+1}\cap BF_ng|= \beta_g|BF_n|\le\beta_g(1+\delta_j)|F_n|.
\end{equation}
Each of the selected components $F_{n_i}c\subset E_g$ is $(F^{-1}_{n_j},\delta_j)$-invariant (each $F_{n_j}$ is symmetric), hence, when multiplied on the left by $F^{-1}_{n_j}$ it can gain at most $\delta_j|F_{n_i}c|$ new elements. Thus the set $E_g$, when multiplied on the left by $F^{-1}_{n_j}$, can gain at most $\delta_j\sum_{F_{n_i}\!c\subset E_g}|F_{n_i}c|$ new elements. On the other hand, denoting by $(F_{n_i}c)^\circ$ the pairwise disjoint sets (contained in respective sets $F_{n_i}c$) as in the definition of $\eps$-disjointness, we also have
$$
\sum_{F_{n_i}\!c\subset E_g}|F_{n_i}c|\le \frac1{1-\eps}\sum_{F_{n_i}\!c\subset E_g}|(F_{n_i}c)^\circ| = \frac1{1-\eps}\Bigl|\bigcup_{F_{n_i}\!c\subset E_g}(F_{n_i}c)^\circ\Bigr|\le \frac1{1-\eps}|E_g|.
$$
Combining this with the preceding statement, we obtain that the set $E_g$, when multiplied on the left by $F^{-1}_{n_j}$, can gain at most $\frac{\delta_j}{1-\eps}|E_g|$ new elements, which is less than $2\delta_j|E_g|$ (we can assume that $\eps<\frac12$).
Denote $H'_{j+1} = F_{n_j}^{-1}H_{j+1}$. By the choice of the components included in $E_g$, the set $F^{-1}_{n_j}E_g$ contains all of $H'_{j+1}\cap F_ng$. Thus, using $(1+2\delta_j)\le (1+\delta_j)^2$ and \eqref{niewiem}, we obtain that
\begin{equation*}
|H'_{j+1}\cap F_ng|\le |F^{-1}_{n_j}E_g| \le (1+2\delta_j)|E_g| \le (1+\delta_j)^3\beta_g|F_n|.
\end{equation*}
Let $N_g = F_ng\setminus H'_{j+1}$. By the above inequality, we know that
\begin{equation}\label{dwa}
|N_g|\ge \bigl(1-(1+\delta_j)^3\beta_g\bigr)|F_n|\ge \bigl(1-(1+\delta_j)^3\beta_g\bigr)\tfrac{|F|}{1+\delta_j},
\end{equation}
where the last inequality follows from $(F_{n_j},\delta_j)$-invariance of $F_n$.

For each $c\in N_g$ we have either $c\in C_j$ and then $\frac{|H_j\cap F_{n_j}c|}{|F_{n_j}|}=1$, or
$\frac{|H_j\cap F_{n_j}c|}{|F_{n_j}|}\ge\eps$ (otherwise $c$ could be added to $C_j$ contradicting its maximality; note that $N_g$ is disjoint from $\bigcup_{i=j+1}^rC_i$). In either case
$|H_j\cap F_{n_j}c|\ge\eps|F_{n_j}|$. This implies that there are at least $\eps|N_g||F_{n_j}|$ pairs $(f,c)$ with $f\in F_{n_j}, c\in N_g$ such that $fc\in H_j$. This in turn implies that there exists at least one $f\in F_{n_j}$ for which
\begin{equation}\label{trzy}
|H_j\cap fN_g|\ge \eps|N_g|.
\end{equation}

Notice that $fN_g$ is contained in $Fg$ (because $N_g\subset F_ng$ and $f\in F_{n_j}$) and disjoint from $H_{j+1}$ ($N_g$ is disjoint from $H'_{j+1}$ which contains $f^{-1}H_{j+1}$). Thus we can estimate, using \eqref{jeden}, \eqref{dwa} and \eqref{trzy}:
\begin{eqnarray*}
\frac{|H_j\cap Fg|}{|F|}&\ge & \frac{|H_{j+1}\cap Fg| + |H_j\cap fN_g|}{|F|}\\
&= & \frac{|H_{j+1}\cap Fg|}{|F|}  + \frac{|H_j\cap fN_g|}{|N_g|}\frac{|N_g|}{|F|}\\
&\ge & \frac{\beta_g-\delta_j}{1+\delta_j} + \eps\frac{1-(1+\delta_j)^3\beta_g}{1+\delta_j}.
\end{eqnarray*}

%Recall that $\underline D_{F_{n_j}F_n}(H_{j+1})\ge \frac1{1+\delta_j}\underline D(H_{j+1})$
%Recall our goal: lower estimate of $\underline D_F(H_j)$. We must estimate from below the ratio $\frac{|H_j\cap Fg|}{|F|}$ for any $g\in G$.
%Given $g$, l

The rest of the proof is elementary calculus. Both terms in the last expression are linear functions of $\beta_g$, the first one
with positive and large slope $\frac1{1+\delta_j}$, the other with negative but small slope $-\eps(1+\delta_j)^2$. Jointly,
the function increases with $\beta_g$. So, we can replace $\beta_g$ by any smaller value, for instance, by $\frac{1-(1-\frac\eps2)^{r-j}}{1+\delta_j}$ (see \eqref{minus} and \eqref{zero}), to obtain
$$
\frac{|H_j\cap Fg|}{|F|}> \frac{1-(1-\tfrac\eps2)^{r-j}}{(1+\delta_j)^2}-\frac{\delta_j}{1+\delta_j} + \eps\bigl(\tfrac1{1+\delta_j}-(1+\delta_j)(1-(1-\tfrac\eps2)^{r-j})\bigr).
$$
Now notice, that if we replace the undivided occurrence of $\eps$ by $\frac{3\eps}4$, we make the entire expression smaller by
some positive value (independent of $g$). On the other hand, if $\delta_j$ is very small and we remove it completely from the
expression, we will perhaps enlarge it, but very little. We now specify $\delta_j$ to be so small, that if we replace $\eps$ by $\frac{3\eps}4$ and remove $\delta_j$ completely, then the expression will become smaller. With such a choice of $\delta_j$ we have
$$
\frac{|H_j\cap Fg|}{|F|}> 1-(1-\tfrac\eps2)^{r-j} +\frac{3\eps}4(1-\tfrac\eps2)^{r-j} = 1-(1-\tfrac\eps2)^{r-j+1} + \xi,
$$
where $\xi>0$ does not depend on $g$. Taking infimum over all $g\in G$ we get, by Lemma \ref{bd},
$$
\underline D(H_j)\ge \underline D_F(H_j) > 1-(1-\tfrac\eps2)^{r-j+1},
$$
and the inductive hypothesis has been derived for $j$.

Once the induction reaches $j=1$ we get that the lower Banach density of $H=H_1$ is larger than $1-(1-\tfrac\eps2)^r$ which,
by the choice of $r$, is larger than $1-\eps$. This concludes the proof of the lemma.
\end{proof}

The next lemma and the following theorem contain our key passage from quasitilings to (exact) tilings.

\begin{lem}\label{disjoint_tilings}
Fix $\eps>0$ and a finite set $K\subset G$. There exists a disjoint $(1-\eps)$-covering quasitiling $\CT^\circ$ of $G$, such that every shape of $\CT^\circ$ is $(K,\eps)$-invariant and
$\h(\CT^\circ)<\eps$.
\end{lem}

\begin{proof}
Let $\xi$ be such that $(1-\xi)^2>1-\eps$ and $\frac{\Theta(\xi)}{1-\xi}\le\frac\eps3$,
where $\Theta(\cdot)$ is defined in Lemma \ref{combinat}. Let $r=r(\xi)$ (as defined in Lemma \ref{quasitilings}) and $\delta=\delta(r,\frac\eps3)$ (as defined in Lemma \ref{entest}).

Let $\CT$ be the quasitiling delivered by Lemma \ref{quasitilings} for $\xi$ (in the role of $\eps$)
and $n_0$ so large that $F_{n}$ is $(K,\xi)$-invariant and $|F_n|>\frac1{\delta(1-\xi)}$ for each $n\ge n_0$.

We will show that the disjoint quasitiling $\CT^\circ$ (as in the definition of $\xi$-disjointness) is good. First of all, by Lemma \ref{cov1}, $\CT^\circ$ is $(1-\xi)^2$-covering (hence also $(1-\eps)$-covering), and by Lemma \ref{ben}, if $\xi$ is small enough, the shapes of $\CT^\circ$ are $(K,\eps)$-invariant. Next, we will verify that $\CT^\circ$ satisfies the assumptions of Lemma \ref{entest} (with $\frac\eps3$ in place of $\eps$).

For $i=1,2,\dots, r$, let $\CS_i$ be the family of shapes of the tiles $T^\circ$ such that $T$ has the shape
$F_{n_i}$. By choosing a subsequence if necessary, we can assume that the sizes of the F\o lner sets satisfy $|F_{n+1}|> 2|F_n|$, which (together with $\xi<\frac12$) ensures that the above families $\CS_i$ are disjoint. The minimal size $s_i$ of a shape in $\CS_i$ is at least $|F_{n_i}|(1-\xi)$ which is larger than $\frac1\delta$, as required. The cardinality of $\CS_i$ is estimated by the number of all $(1-\xi)$-subsets of $F_{n_i}$ (the new center for every such subset is determined by Remark \ref{centers}), that is, by $2^{\Theta(\xi)|F_{n_i}|}\le 2^{\frac{\Theta(\xi)}{1-\xi}s_i}\le 2^{\frac\eps3 s_i}$. Now the application of Lemma \ref{entest} ends the proof.
\end{proof}

\begin{thm}\label{exact_tilings}
Fix $\eps>0$ and a finite set $K\subset G$. There exists an (exact) tiling $\CT^*$ of $G$, such that every shape of $\CT^*$ is $(K,\eps)$-invariant, and $\h(\CT^*)<\eps$.
\end{thm}

\begin{proof} Let $\CT^\circ$ be the disjoint quasiltiling delivered by the preceding lemma for the parameters $\gamma$ and $K$, where $\gamma<\min\{\frac12,\frac\eps2,\frac\delta6\}$ with $\delta$ specified
in Lemma \ref{ben} for $\eps$ and $K$. In the following steps of the construction we will modify this quasitiling so it becomes a tiling (i.e., it will cover all of~$G$).

In every shape $S$ of $\CT^\circ$ we choose two disjoint subsets, $A(S)$ and $A'(S)$, each of cardinality $\lceil2\gamma\abs{S}\rceil$ (which, we can assume, is smaller than $3\gamma|S|$). Next, if $T^\circ=Sc$ is a tile of $\CT^\circ$, we let $A(T^\circ)=A(S)c$ (and analogously for $A'(T^\circ)$). The unions of these latter sets over the entire tiling yield two disjoint sets $A$ and $A'$, each having lower Banach density at least $2\gamma(1-\gamma)>\gamma$ (see Lemma \ref{cov1}). Let $B$ denote the set of elements of $G$ that are not covered by the tiles of $\CT^\circ$. Directly, since $\CT^\circ$ is $(1-\gamma)$-covering, the upper Banach density of $B$ is less than $\gamma$.

%Our goal is to attach the elements of $B$ to the tiles of $\CT^\circ$ in a way that will not increase the entropy by more than $\frac{2\gamma}{3}$ (the density of the centers will be unchanged even if we add new elements to the tiles, but the number of possible shapes will increase, causing an increase in entropy). To this end, we will construct an injective mapping $\Phi:B\to A\cup A'$ and we will attach every $b\in B$ to the tile that contains $\Phi(B)$. This will increase every tile by a proportion of at most $5\eta$, and since we arrange that the elements $\Phi(b)b^{-1}$ range over a finite set $F$, there will be only a finite number of possible new shapes (since every tile $T$ can only acquire new elements that are within $FT$).

Let $F$ be a finite, symmetric subset of $G$ such that the proportion of elements of $A$ in any translate $Fg$ is at least $\gamma$, the same holds for $A'$, and the proportion of elements of $B$ in $Fg$ is less than $\gamma$. Let $\xi<\frac\eps4$ be so small that any symbolic dynamical system with the alphabet $F\cup\{0\}$ and with the upper Banach density of non-zero symbols smaller than $\xi$ has entropy less than $\frac\eps4$ (see Lemma \ref{symbolic_small_entropy}).

Using Lemma \ref{disjoint_tilings} again (with different parameters), we obtain a disjoint, $(1-\xi)$-covering quasitiling $\CT'$ with entropy less than $\xi$ (hence less than $\frac\eps4$). Moreover, we can assume an ``improved disjointness'': if $T'_1$ and $T'_2$ are different tiles of $\CT'$ then $FT'_1\cap FT'_2=\emptyset$; this can be achieved (using Lemmas \ref{estim} and \ref{cov1}) by requesting the disjoint quasitiling to be $(1-\xi')$-covering, with $(F,\xi')$-invariant tiles and entropy less than $\xi'$ (for a small $\xi'$), and then replacing the tiles by their $F$-cores with the centers determined by Remark \ref{centers} (such modification does not increase the entropy because it produces a topological factor by a finite horizon algorithm).

Let $T'$ be a tile of $\CT'$. We define a relation $R$ between $B\cap T'$ and $A\cap FT'$: $bRa$ if and only if $a\in Fb$. By the definition of $F$, for every $b$ there are at least $\gamma\abs{F}$ elements $a$ such that $bRa$, and for every $a$ there are at most $\gamma\abs{F}$ elements $b$ such that $bRa$. By the marriage theorem (Theorem \ref{marriage}), there exists an injective mapping $\phi_{T'}$ from $B\cap T'$ into $A\cap FT'$ such that $\phi_{T'}(b)\in Fb$ for every $b$ in the domain. The ``improved disjointness'' implies that not only domains, but also images, of the maps $\phi_{T'}$ are disjoint, so that uniting the graphs of $\phi_{T'}$ we obtain an injective map $\phi:B\cap\bigcup\CT'\to A$. Moreover, we can arrange that whenever $T'=Sc$ and $T''=Sc'$ (i.e., two tiles of $\CT'$ have the same shape $S$), $B\cap T'' = (B\cap T')c^{-1}c'$ and $A\cap FT'' = (A\cap FT')c^{-1}c'$, then $\phi_{T''}(b) = \phi_{T'}(b{c'}^{-1}c)c^{-1}c'$ (for every $b\in B\cap T''$), i.e., that $\phi_{T'}$ depends only on how $T'$ and $FT'$ contain and intersect the tiles of $\CT^\circ$. In this manner, the map $\phi$ is determined by $\CT'$ and $\CT^\circ$ via a finite horizon algorithm.

Further, let $B'$ be the remaining part of $B$ (not covered by the tiles of $\CT'$). Again, we define a relation (which we will again denote by $R$) between $B'$ and $A'$, by the same formula: $bRa$ if and only if $a\in Fb$. As before, by the definition of $F$, the assumptions of the marriage theorem are fulfilled, yielding another injective mapping $\phi'$ from $B'$ into $A'$ with $\phi'(b\,')\subset Fb\,'$ (this map, however, is not necessarily determined by a finite horizon algorithm). Uniting the maps $\phi$ and $\phi'$ (in terms of uniting their graphs) we obtain an injective mapping $\Phi$ from $B$ into $A\cup A'$, with the property that for every $b$, $\Phi(b)\in Fb$.

We can now define the desired tiling $\CT^*$: every tile of $\CT^*$ will have the form $T^* = T^\circ\cup \set{b\in B: \Phi(b)\in T^\circ}$ for some $T^\circ\in \CT^\circ$. We define the center of this new tile to be the same as the center for $T^\circ$. Each shape of $\CT^\circ$ may produce many new shapes of $\CT^*$, however, since $\set{b\in B: \Phi(b)\in T^\circ}\subset FT^\circ$, the variety of new shapes remains finite.
The center sets for each new shape are then determined automatically. By the construction of $\Phi$, the tiles of $\CT^*$ are disjoint and cover all of $G$. The added set $\set{b\in B: \Phi(b)\in T^\circ}$ has cardinality at most that of $A(T^\circ)\cup A'(T^\circ)$, hence $|T^*|\le |T^\circ|(1+6\gamma)$. Therefore, by Lemma~\ref{ben} (and the selection of $\gamma$), $T^*$ is $(K,\eps)$-invariant.

It remains to show that $\CT^*$ has entropy strictly less than $\eps$. Consider the symbolic element $y\in (F\cup\{0\})^G$ defined as follows: $y(b)=g\in F$ if $b\in B'$ and $\phi'(b)=gb$, and $y(b)=0$ for $b\notin B'$. Since $B'$ has upper Banach density less than $\xi$, the upper Banach density of non-zero symbols in $y$ is also less than $\xi$. Thus the topological entropy of $y$ is less than $\frac\eps4$. Now observe that the tiling $\CT^*$ is determined by the quasitilings $\CT^\circ$, $\CT'$ and the contents of $y$, via a finite horizon algorithm ($y$ is not determined by $\CT^\circ$, $\CT'$ via a finite horizon algorithm, but once we acquire the information coming from $y$, the finite horizon statement holds). Thus $\CT^*$ is a topological factor of a topological joining between $\CT^\circ$ and $\CT'$, and $y$. Therefore the entropy of $\CT^*$ is indeed less than $\gamma+\frac\eps4+\frac\eps4<\eps$.
\end{proof}

\section{A congruent sequence of tilings with entropy zero}

In this section we strengthen the preceding result by obtaining exact tilings (with arbitrarily ``well-invariant'' shapes) which have topological entropy zero. This is done in two steps, via constructing a sequence of exact tilings $(\tilde\CT_k)_{k\ge1}$ with entropies tending to zero, and which is  \emph{congruent},  i.e., such that, for each $k\ge 1$, every tile of $\tilde\CT_{k+1}$ equals a union of tiles of $\tilde\CT_k$. Next, we transform this sequence into $(\overline\CT_k)_{k\ge1}$, in which every tiling is a topological factor of its successor, and hence all of them have entropy zero.

\begin{lem}\label{sqt}
Fix a converging to zero sequence $\eps_k> 0$ and a sequence $K_k$ of finite subsets of $G$.
There exists a congruent sequence of tilings $\tilde\CT_k$ of $G$ such that the shapes of
$\tilde\CT_k$ are $(K_k,\eps_k)$-invariant and $\h(\tilde\CT_k)<\eps_k$.
\end{lem}
\begin{proof}
Use Theorem~\ref{exact_tilings} to obtain a tiling $\CT^*_1$ whose shapes are $(K_1,\eps_1)$-invariant and topological entropy is strictly less than $\eps_1$. We set $\tilde\CT_1=\CT_1^*$. Suppose a tiling $\tilde\CT_k$
(as in the formulation of the lemma) has been constructed. We will construct $\tilde\CT_{k+1}$.

Let us denote by $D_k$ the set $\bigcup \CS(\tilde\CT_k)$. By an application of Lemmas \ref{ben} and \ref{symbolic_small_entropy}, there exists $\delta>0$ such that whenever $\frac{|T'\triangle T|}{|T|}<\delta$ and $T$ is $(K_{k+1},\delta)$-invariant then $T'$ is $(K_{k+1},\eps_{k+1})$-invariant, and any symbolic dynamical system with the alphabet $\CS(\tilde\CT_k)\cup\{0\}$ and upper Banach density of non-zero symbols smaller than $\delta$ has topological entropy less than $\frac{\eps_{k+1}}{2}$. Choose $\delta_k<\min\{\frac{\eps_{k+1}}2,\frac\delta{|D_k|+1}\}$. We can now use Theorem~\ref{exact_tilings}
again, with the parameter $\delta_k$, to obtain a tiling $\CT^*_{k+1}$ with entropy less than $\delta_k$ (hence strictly less than $\frac{\eps_{k+1}}{2}$), and the shapes of which are $(K'_{k+1},\delta_k)$-invariant, where $K'_{k+1} = K_{k+1}\cup D_k\cup D_k^{-1}$.\footnote{Such shapes are also $(D_k,\delta_k)$-invariant, $(D^{-1}_k,\delta_k)$-invariant, but only $(K_{k+1},2\delta_k)$-invariant; it is so since $D_k$ contains the unity, which we do not assume about $K_{k+1}$.} We need to modify the tiling $\CT^*_{k+1}$ to make it congruent with $\tilde\CT_k$, i.e., ensure that its tiles are unions of the tiles of $\tilde\CT_k$. Define a ``modification map'' $T^*\mapsto \tilde T$ (where $T^*\in\CT^*_{k+1}$) by $\tilde T=\bigcup\set{Sc\in \tilde\CT_k:c\in T^*}$. The center of $\tilde T$ is determined according to Remark \ref{centers}. That way we create a modified tiling, denoted $\tilde\CT_{k+1}$, congruent with $\tilde\CT_k$. It is easily verified that each tile $\tilde T$ of $\tilde\CT_{k+1}$ satisfies
$$
T^*_{D_k^{-1}}\subset \tilde T\subset D_kT^*
$$
and, clearly, $T^*$ is located between the same two extreme sets, hence
$$
\tilde T\triangle T^*\subset D_kT^*\setminus T^*_{D_k^{-1}}.
$$
Since $T^*$ is $(D_k,\delta_k)$-invariant (and $D_k$ contains the unity), we have $|D_kT^*|\le (1+\delta_k)|T^*|$. Also, since $T^*$ is
$(D^{-1}_k,\delta_k)$-invariant, $T^*_{D_k^{-1}}$ is a $(1-|D_k|\delta_k)$-subset of $T^*$ (see Lemma \ref{estim}). This yields
$$
\frac{|\tilde T\triangle T^*|}{|T^*|}\le (|D_k|+1)\delta_k< \delta.
$$
Since $T^*$ is $(K_{k+1},2\delta_k)$-invariant and $2\delta_k\le \delta$, the selection of $\delta$ implies $(K_{k+1},\eps_{k+1})$-invariance of $\tilde T$.

We will now argue that the modification does not increase the entropy too much. In the argument below we will refer to the tiles of $\tilde\CT_k$ as ``small'', and the tiles of $\CT^*_{k+1}$ as ``large''.

We claim, that in order to determine the modified large tiles, \emph{in addition} to knowing $\CT^*_{k+1}$, we only need to examine the centers of the small tiles lying \emph{outside} the union of the $D_k$-cores of the large tiles. Indeed, after all such centers have been examined and their corresponding small tiles have been allocated among the large tiles, the remaining part of each large tile $T^*$ (not covered by the above small tiles) can be ``blindly'' included to $\tilde T$; we do not need to check where exactly the remaining centers of small tiles are. It is so because a point of $T^*$ does not belong to $\tilde T$ only if it belongs to a small tile with center in a different large tile, say ${T'}^*$. In such case, however, this center does not belong to the $D_k$-core of ${T'}^*$, hence it lies outside the union of all such cores, and such centers have been already examined. So, the necessary information (additional to knowing $\CT^*_{k+1}$) can be encoded in a symbolic element obtained from the symbolic representation of $\tilde\CT_k$ (with non-zero symbols at the centers of the tiles) in which all symbols \emph{inside} the above mentioned $D_k$-cores are ignored, i.e., replaced by zeros. Since the union of these cores has lower Banach density at least $1-\delta$ (because $T_{D_k}^*$ is a $(1- |D_k| \delta_k)$-subset of $T^*$, Lemma \ref{cov1} applies), the upper Banach density of non-zero symbols in the discussed symbolic element is at most $\delta$. Its alphabet is $\CS(\tilde\CT_k)\cup\{0\}$, hence, by the choice of $\delta$, the entropy of such a symbolic element is less than $\frac{\eps_{k+1}}2$. Adding the entropy of $\CT^*_{k+1}$ we get that the entropy of $\tilde\CT_{k+1}$ is strictly less than $\eps_{k+1}$.
\end{proof}

The next statement is perhaps the most important in this paper.

\begin{thm}\label{congruent}
Let $G$ be an infinite countable amenable group.
Fix a converging to zero sequence $\eps_k> 0$ and a sequence $K_k$ of finite subsets of $G$.
There exists a congruent sequence of (exact) tilings $\overline\CT_k$ of $G$ such that the shapes of
$\overline\CT_k$ are $(K_k,\eps_k)$-invariant and $\h(\overline\CT_k)=0$ for each $k$.
\end{thm}
\begin{proof}
We have constructed a congruent sequence of tilings $\tilde\CT_k$, each of entropy strictly less than $\eps_k$. We need to modify them one more time, to kill their entropy completely. To this end we will need another inductive procedure concluded by a limit passage.

First of all, in the construction of the sequence $\tilde\CT_k$ we add one more inductive (easily fulfilled) requirement: The tiling $\tilde\CT_k$ has entropy strictly less than $\eps_k$, hence there exists a finite set $E_k$ (for example a far enough member of the F\o lner sequence) such that the $E_k$-complexity of $\tilde\CT_k$ (i.e., the number of all blocks of the form $g\tilde x|_{E_k}$, where $\tilde x$ is the symbolic representation of $\tilde\CT_k$) does not exceed $2^{\eps_k|E_k|}$. For later purposes, we assume also that $|E_k|>\frac1{\eps_k}$.
We require that all shapes of $\tilde\CT_{k+1}$, in addition to being $(K_{k+1},\eps_{k+1})$-invariant, are also
$(E_k,\delta_k)$-invariant, for $\delta_k=\frac{\eps_k}{|E_k|\log{(|\Lambda_k|)}}$, where $\Lambda_k=\CS(\tilde\CT_k)\cup\{0\}$ is the alphabet used by symbolic representation of $\tilde\CT_k$. We can start the induction.

Every tile of $\tilde\CT_2$ is a union of the tiles of $\tilde\CT_1$, thus every shape of $\tilde\CT_2$ is partitioned by (shifted) shapes of $\tilde\CT_1$. However, for each shape of $\tilde\CT_2$ there possibly occur more than one different ways it is partitioned. Now, for each shape $S$ of $\tilde\CT_2$ we select one such way and call it ``the master partition'' of $S$. We create a new tiling, inscribed in $\tilde\CT_2$, using the same family of shapes as $\tilde\CT_1$ (perhaps not all shapes will be used, but that does not bother us), as follows: in each tile of $\tilde\CT_2$ we apply the (appropriately shifted) master partition of its shape. We denote this new tiling by $\CT_1^{(2)}$. Notice that this tiling is completely determined by $\tilde\CT_2$ and the ``finite bank of information'' containing the master partitions of all shapes of $\tilde\CT_2$. And clearly, the coding from $\tilde\CT_2$ to $\CT_1^{(2)}$ has a finite horizon. Thus, $\CT_1^{(2)}$ is a topological factor of $\tilde\CT_2$ (and congruent with it).

Analogously, from $\tilde\CT_2$ and $\tilde\CT_3$ we create a tiling $\CT_2^{(3)}$ which
uses the same shapes as $\tilde\CT_2$ and is congruent with and a topological factor of $\tilde\CT_3$.
Now, applying to the tiles of $\CT_2^{(3)}$ the master partitions from the preceding step (by the shifted shapes of $\tilde\CT_1$) we also create a new tiling $\CT_1^{(3)}$ using the same shapes as $\tilde\CT_1$, congruent with and being a topological factor of $\CT_2^{(3)}$ (and $\tilde\CT_3$).

Continuing in an obvious way we create a triangular array of tilings $\CT_k^{(j)}$
($k\le j$; we also place $\tilde\CT_k$ as $\CT_k^{(k)}$ along the diagonal of that array),
where $k$ is the row number, $j$ is the column number, the rows are finite and the columns are infinite,
with the following properties:

\begin{enumerate}
	\item $\CT_k^{(j)}$ uses the same shapes as $\tilde\CT_k$, for every $k\le j$;
	\item $\CT_k^{(j)}$ is congruent with and a topological factor of $\CT_l^{(j)}$, whenever $k\le l\le j$;
	\item each tile of $\CT_{k+1}^{(j)}$ is partitioned by the tiles of $\CT_k^{(j)}$ according to the master partition of its shape $S$ (here $k<j$).
\end{enumerate}
We recall that the above master partition is defined using the ``original'' tilings $\tilde\CT_k$ and $\tilde\CT_{k+1}$ (as a selected one of many ways of partitioning the shape $S$) and then it does not change in the following steps of the construction of the array of tilings.

By compactness of the symbolic spaces $\Lambda_k^G$ (where $\Lambda_k$ is the alphabet used in all tilings in the $k$th column), there exists a subsequence $j_i$ such that $\CT_k^{(j_i)}$ converges, for every $k$, to some symbolic element, say $\overline\CT_k\in\Lambda_k^G$. Now, all combinatorial properties satisfied by the elements $\CT_k^{(j_i)}$ (and pairs $\CT_k^{(j_i)}$, $\CT_l^{(j_i)}$) verifiable by finite horizon testing (where the horizon does not depend on $j_i$) pass on to the limit element (because such properties hold on closed sets). In particular:
\begin{enumerate}
	\item $\overline\CT_k$ represents an exact tiling with shapes $\CS(\tilde\CT_k)$;
	\item $\overline\CT_k$ is congruent with and a topological factor of $\overline\CT_l$, whenever $k\le l$;
	\item each tile of $\overline\CT_{k+1}$ is partitioned by the tiles of $\overline\CT_k$ according to the master partition of its shape $S$ (for any $k\ge 1$).
\end{enumerate}
Property (1) implies that the shapes of $\overline\CT_k$ are $(K_k,\eps_k)$-invariant, as needed.

Because the estimation of the topological entropy of $\overline\CT_k$ involves measure-theoretic entropy (and the variational principle) we isolate it as a separate lemma. The main proof will be resumed afterwards.

\begin{lem}For each $k$ and every invariant measure $\mu$ supported by the dynamical tiling $\overline X$ generated by $\overline\CT_k$ we have $h_\mu(G) <4 \eps_k$.
\end{lem}
\begin{proof}
Recall that the measure-theoretic entropy is computed as
$$
h_\mu(G) = \lim_{n\to\infty} \frac1{|F_n|}H_\mu(\overline X_{F_n}),
$$
where
$$
H_\mu(\overline X_{F_n}) = -\sum_{B\in \overline X_{F_n}}\mu([B])\log(\mu([B])).
$$
Moreover, the above limit is the same as the infimum (by the strong sub-additivity property of entropy function, see \cite[Proposition 3.1.9]{MO}). Hence, in order to estimate $h_\mu(G)$ from above it suffices to estimate $\frac1{|F_n|}H_\mu(\overline X_{F_n})$ for just one (arbitrary) F\o lner set. We will use the particular set $E_k$ selected at the beginning of the last inductive construction, such that the $E_k$-complexity of $\tilde\CT_k$ does not exceed $2^{\eps_k|E_k|}$. We have
$$
H_\mu(\overline X_{E_k}) = -\!\!\!\!\!\!\sum_{B\in \overline X_{E_k}\cap\tilde X_{E_k}}\!\!\!\!\!\!\mu([B])\log(\mu([B]))
-\!\!\!\!\!\!\sum_{B\in \overline X_{E_k}\setminus\tilde X_{E_k}}\!\!\!\!\!\!\mu([B])\log(\mu([B])),
$$
where $\tilde X$ is the dynamical tiling generated by $\tilde\CT_k$.
The entropy of a finite-dimensional sub-probabilistic vector is estimated from above by the mass of the vector times the logarithm of its dimension, plus 1. Thus, the first sum does not exceed 1 times the logarithm of the $E_k$-complexity of $\tilde\CT_k$ (i.e., $\eps_k|E_k|$), plus 1. The second sum does not exceed the measure of the union of all cylinders corresponding to blocks $B$ with domain $E_k$ occurring in $\overline X$ but not in $\tilde X$, times the logarithm of the number of all possible blocks with domain $E_k$ (i.e., times $\log(|\Lambda_k|^{|E_k|})$), plus 1.
Observe that if $g$ is such that $E_kg$ is contained in a tile of $\overline\CT_{k+1}$ then the associated block $B$ (formally equal to $g\overline\CT_k|_{E_k}$) arises from the master partition of this tile's shape and thus the same block $B$ occurs in $\tilde\CT_k$ (at some position $g'$), hence in $\tilde X$. So, a block $B$ occurs in $\overline X$ but not in $\tilde X$ only if it occurs in $\overline\CT_k$ exclusively at such positions $g$ that $E_kg$ is not contained in one tile of $\overline \CT_{k+1}$. This happens only when $g$ does not fall in the $E_k$-core of any tile of $\overline \CT_{k+1}$. Recall that each tile of $\overline \CT_{k+1}$ is $(E_k,\delta_k)$-invariant, so, by Lemma \ref{estim}, its $E_k$-core is its $(1-\xi)$-subset, where $\xi = \delta_k|E_k|= \frac{\eps_k}{\log(|\Lambda_k|)}$. Now, by Lemma \ref{cov1} (for a $1$-covering tiling) we get that the upper Banach density of the set not covered by the discussed $E_k$-cores is at most $\xi$. By Lemma~\ref{measure}, the set of all points in the dynamical tiling generated by $\overline \CT_{k+1}$ (each such point represents a tiling), such that $e$ does not belong to the union of the $E_k$-cores of all tiles, has measure at most $\xi$, for every invariant measure supported by this dynamical tiling. It follows from our earlier discussion, that the above set contains the preimage (via the factor map from $\overline\CT_{k+1}$ to $\overline\CT_k$), of the union of the cylinders $B$ indexing the second large sum above. Thus any invariant measure supported by $\overline\CT_k$ (in particular $\mu$) gives this union a value at most $\xi$.
Eventually, we get the estimate
$$
H_\mu(\overline X_{E_k})\le \eps_k|E_k| + \xi|E_k|\log(|\Lambda_k|) + 2 = 2\eps_k|E_k| + 2<4\eps_k|E_k|.
$$
\end{proof}

We return to the main proof. The above lemma, together with the variational principle for amenable group actions (see \cite[Variational Principle 5.2.7]{MO}) imply that $\h(\overline\CT_k)\le 4\eps_k$. On the other hand, since $\overline\CT_k$ is a factor of any $\overline\CT_j$ with $j\ge k$, we obtain $\h(\overline\CT_k)\le 4\eps_j$, which implies that $\h(\overline\CT_k)=0$.
\end{proof}

\section{Free action with entropy zero}

We are in a position to construct a free, zero entropy action of $G$ on a zero-dimensional space.

\begin{thm} \label{free action}
Let $G$ be an infinite countable amenable group. There exists a zero-dimensional space $\mathfrak X$
and a free action of $G$ on $\mathfrak X$ which has topological entropy zero.
\end{thm}

\begin{proof}
It suffices to show that for every $g\in G$, $g\neq e$, there exits a symbolic system $\mathfrak{X}_g\subset \{0,1\}^G$ with topological entropy zero and such that no points of $\mathfrak{X}_g$ are fixed by the shift by $g$. Once this is done, we can define $\mathfrak{X}=\prod_{g\neq e}\mathfrak{X}_g$ (with the product action). This system obviously has no points fixed by any $g\neq e$ (i.e., this is a free action), and as a countable product of zero-entropy subshifts it is zero-dimensional and has topological entropy zero.
\smallskip

We will use two different techniques, depending on whether $g$ has a finite order or not. The finite order case strongly relies upon our exact tilings with entropy zero constructed in the preceding sections. In each case the alphabet of $\mathfrak X_g$ will consist of two symbols (although not necessarily denoted as 0 and 1).
\smallskip

Fix $g\in G$ and assume the order of $g$ to be infinite. The following is an equivalence relation on $G$: $f \sim h \iff f=g^ph$ for some $p\in \mathbb{Z}$. Let $B$ be a set containing exactly one element from each equivalence class. Now every element $h\in G$ has a unique representation $h=g^pb$, where $p\in \mathbb{Z}$ and $b\in B$. Denote this exponent $p$ by $p(h)$. We define a symbolic element $x \in \set{-1,1}^G$ by $x(h)=(-1)^{p(h)}$, and we let $\mathfrak{X}_g$ be the shift orbit closure of $x$.

Suppose that for some $y\in \mathfrak{X}_g$ we have $gy=y$, in particular $y(g)=y(e)$. Let $h_n$ be a sequence of elements of $G$ such that $y=\lim_{n\to\infty}h_nx$. For large enough $n$ we have $y(e)=x(h_n)$ and $y(g)=h_nx(g)=x(gh_n)$. Therefore, on one hand $x(gh_n)=x(h_n)$, and on the other, by the definition of $x$, we have $x(gh_n)= -x(h_n)$. We have shown that $g$ fixes no points of $\mathfrak X_g$.

To show that $\h(\mathfrak{X}_g)=0$ let $F=\{g,g^2,\dots,g^n\}$ for some $n\in\N$. It is easy to see that the $F$-complexity equals 2 (there are only two blocks with domain $F$: $[-1,1,\dots,(-1)^n]$ and $[1,-1,\dots,(-1)^{n+1}]$). Thus $\frac1{|F|}\log N(F) = \frac{\log 2}{n}$, which is arbitrarily small, implying, via Theorem~\ref{infimum}, that the entropy of the subshift is indeed zero.
\smallskip

Now assume the order of $g$ is finite and equals $q$. We can still define the relation $\sim$ as before, the only difference being that the equivalence classes are now finite. Therefore they form a tiling, say $\CT_0$, of $G$ (this tiling has one shape $S=\{e,g,\dots,g^{q-1}\}$, the centers are assigned arbitrarily within the tiles\footnote{Note that any point within the tile can be assigned its center without needing to shift the shape.}). Setting $K_1=\{e\}$ and $\eps_1=1$ we see that $\CT_0$ can be used (in place of $\CT_1^*$) as the first tiling $\widetilde\CT_1$ in Lemma~\ref{sqt}. Now Theorem~\ref{congruent} produces a new sequence of tilings $(\overline\CT_k)_{k\geq 1}$ such that $\h(\overline\CT_k)=0$ for each $k\ge 1$ and $\overline\CT_k$ uses the same shapes as $\widetilde\CT_k$. In particular, $\overline\CT_1$ has the same one shape $S$, i.e., it is the partition into equivalence classes ($\overline\CT_1$ differs from $\widetilde\CT_1=\CT_0$ in having the centers positioned ``more intelligently'' within the tiles). Let $\mathfrak{X}_g$ be the dynamical tiling generated by $\overline\CT_1$.
We already know that this symbolic system has entropy zero. Recall that every symbolic element $y\in\mathfrak X_g$ represents a tiling (using the same one shape $S$), in particular every tile has only one center, i.e., within every tile there is only one nonzero symbol $S$ (we agreed
to use shape labels as symbols placed at the tile centers, and zeros everywhere else).

Suppose that for some $y\in \mathfrak{X}_g$ we have $gy=y$. Let $c$ be the center of the tile containing $e$. Then $c=g^p$ for some $p\in\{0,1,\dots,q-1\}$ and $y(c)=S$. Since $gy=y$ we also have $S=gy(c)=y(cg) = y(g^{p+1})$. Clearly $g^{p+1}$ belongs to the same class (i.e., the same tile) as $c$, and since $g\neq e$,
$g^{p+1}\neq c$. We have found two nonzero symbols in one tile, a contradiction. This concludes the proof.
\end{proof}
%\begin{figure}
%% Use the relevant command to insert your figure file.
%% For example, with the graphicx package use
  %\includegraphics{example.eps}
%\end{figure}
%%
%% For two-column wide figures use
%\begin{figure*}
%% Use the relevant command to insert your figure file.
%% For example, with the graphicx package use
  %\includegraphics[width=0.75\textwidth]{example.eps}
%\end{figure*}
%\begin{table}
%\begin{tabular}{lll}
%\hline\noalign{\smallskip}
%\end{tabular}
%\end{table}
%

\begin{acknowledgements}
Part of the work was carried out during a series of visit of T. Downarowicz to School of Mathematical Sciences and LMNS (Fudan University) and a visit of G. H. Zhang to Institute of Mathematics, Polish Academy of Sciences (IMPAN). We gratefully acknowledge the hospitality of Fudan University and IMPAN.
The research of the first two authors is funded by NCN grant 2013/08/A/ST1/00275.
G. H. Zhang is supported by FANEDD (201018) and NSFC (11271078).

\end{acknowledgements}


\begin{thebibliography}{}
\bibitem[BBF]{BBF}
Mathias Beiglb{\"o}ck, Vitaly Bergelson, and Alexander Fish, Sumset phenomenon in countable amenable groups, Adv. Math. 223 no.~2, 416--432 (2010). MR 2565535 (2011a:11011)
\bibitem[BD]{BD}
Mike Boyle and Tomasz Downarowicz, The entropy theory of symbolic
  extensions, Invent. Math. 156, no.~1, 119--161 (2004). MR 2047659
  (2005d:37015)

\bibitem[Da]{Da}
Alexandre~I. Danilenko, Entropy theory from the orbital point of view,
  Monatsh. Math. 134, no.~2, 121--141  (2001). MR 1878075 (2002j:37011)

\bibitem[Do]{Do}
Tomasz Downarowicz, Entropy in dynamical systems, New Mathematical
  Monographs, vol.~18, Cambridge University Press, Cambridge (2011). MR 2809170(2012k:37001)

\bibitem[DF]{DF}
Tomasz Downarowicz and Bartosz Frej, Entropy as infimum over finite subsets of the acting group, preprint.

%{\color{red}
%\bibitem[DH]{DH}
%Tomasz Downarowicz and Dawid Huczek, \emph{Faithful zero-dimensional principal
%  extensions}, Studia Math. 212 (2012), no.~1, 1--19. 3004163
%}

%{\color{blue} This reference is not cited in the paper.}

%{\color{red}
%\bibitem[H]{H} Dawid Huczek, \emph{Zero-dimensional extensions of amenable
%group actions}, (preprint)

%{\color{blue} Please update the state of this paper.}

\bibitem[L]{L}
Elon Lindenstrauss, Pointwise theorems for amenable groups, Invent.
  Math. 146, no.~2, 259--295 (2001). MR 1865397 (2002h:37005)

\bibitem[LW]{LW}
Elon Lindenstrauss and Benjamin Weiss, Mean topological dimension,
  Israel J. Math. 115, 1--24 (2000). MR 1749670 (2000m:37018)

\bibitem[MO]{MO}
Jean Moulin~Ollagnier, Ergodic theory and statistical mechanics, Lecture
  Notes in Mathematics, vol. 1115, Springer-Verlag, Berlin (1985). MR 781932 (86h:28013)

\bibitem[N]{N}
I.~Namioka, F\o lner's conditions for amenable semi-groups, Math. Scand.
  15, 18--28 (1964). MR 0180832 (31 \#5062)

\bibitem[OW1]{OW1}
Donald~S. Ornstein and Benjamin Weiss, Ergodic theory of amenable group
  actions. {I}. {T}he {R}ohlin lemma, Bull. Amer. Math. Soc. (N.S.) 2, no.~1, 161--164 (1980). MR 551753 (80j:28031)

\bibitem[OW2]{OW2}
Donald~S. Ornstein and Benjamin Weiss, Entropy and isomorphism theorems for actions of amenable groups, J. Analyse Math. 48, 1--141 (1987). MR 910005 (88j:28014)

\bibitem[RW]{RW}
Daniel~J. Rudolph and Benjamin Weiss, \emph{Entropy and mixing for amenable
  group actions}, Ann. of Math. (2) 151, no.~3, 1119--1150  (2000).
  MR 1779565 (2001g:37001)

\bibitem[We]{We}
Benjamin Weiss, Monotileable amenable groups, Topology, ergodic theory,
  real algebraic geometry, Amer. Math. Soc. Transl. Ser. 2, vol. 202, Amer.
  Math. Soc., Providence, RI, pp.~257--262 (2001). MR 1819193 (2001m:22014)

\bibitem[WZ]{WZ}
Thomas Ward and Qing Zhang, The {A}bramov-{R}okhlin entropy addition
  formula for amenable group actions, Monatsh. Math. 114,
  no.~3-4, 317--329 (1992). MR 1203977 (93m:28023)


% \bibitem{K}
%  U.Krengel,
%  \emph{Ergodic Theorems},
%  Walter de Gruyter, Berlin, New York (1985).

%\bibitem{Pa}
%	K.R. Parthasarathy, \emph{Probability Measures On Metric Spaces}, Academic Press, New York (1967).

%\bibitem{P}
%  R.R.Phelps,
%  \emph{Lectures on Choquet's Theorem},
%  D. Van Nostrand Company, Princeton, New Jersey (1966).

\end{thebibliography}


\end{document}